\documentclass[11pt]{article}
\usepackage{amssymb,amsmath,latexsym,amscd,pb-diagram}

\newtheorem{theorem}[equation]{Theorem}
\newtheorem{corollary}[equation]{Corollary}
\newtheorem{lemma}[equation]{Lemma}
\newtheorem{proposition}[equation]{Proposition}
\newtheorem{definition}[equation]{Definition}
\newtheorem{remark}[equation]{Remark}
\numberwithin{equation}{section}

\newcommand{\ol}{\overline}
\newcommand{\Z}{\mathbb{Z}}
\newcommand{\KK}{\mathbb{K}}
\newcommand{\LL}{\mathbb{L}}

\newcommand{\fg}{{\mathfrak g}}
\newcommand{\fb}{{\mathfrak b}}
\newcommand{\fh}{{\mathfrak h}}

\newcommand{\qed}{\hfill $\Box$}

\newcommand{\bmu}{{\mathbf \mu}}
\newcommand{\bB}{{\bf B}}
\newcommand{\bT}{{\bf T}}
\newcommand{\bH}{{\bf H}}
\newcommand{\bG}{{\bf G}}
\newcommand{\bC}{{\bf C}}

\newcommand{\Ad}{{\rm Ad}}

\newcommand{\Gal}{{\mathcal Gal}}

\newcommand\limind{\mathop{\oalign{lim\cr
\hidewidth$\longrightarrow$\hidewidth\cr}}}

\newcommand\bOut{\text{\bf Out}}

\newcommand{\bSL}{{\mathbf{SL}}}
\newcommand{\bGL}{{\mathbf{GL}}}
\newcommand{\bPGL}{{\mathbf{PGL}}}
\newcommand{\bAut}{{\mathbf{Aut}}}

\begin{document}
\title{Belavin--Drinfeld quantum groups and Lie bialgebras: Galois cohomology considerations}
\author{Arturo Pianzola$^{1,2}$ and Alexander Stolin$^{3}$}
\maketitle

$^1$ Department of Mathematical \& Statistical Science, University of Alberta, Edmonton, Alberta, T6G 2G1, Canada.

$^2$ Centro de Altos Estudios en Ciencias Exactas, Avenida de Mayo 866, (1084), Buenos Aires, Argentina.

$^3$ Department of Mathematical Sciences, Chalmers University of Technology and the University of Gothenburg, 412 96 Gothenburg, Sweden.

\begin{abstract} We relate the Belavin--Drinfeld cohomologies (twisted and untwisted) 
that have been introduced in the literature to study certain families of quantum groups and Lie bialgebras
over a non algebraically closed field $\KK$ of characteristic 0 to the standard 
non-abelian Galois cohomology $H^1(\KK, \bH)$ for a suitable algebraic $\KK$-group $\bH.$ 
The approach presented allows us to establish in full generality certain conjectures that were known to hold for the classical types of the
split simple Lie algebras. 
\end{abstract}

\vskip.25truein \section{Introduction}
The appearance of Galois cohomology in the classification of certain quantum groups is one of the primary goals of this paper. In order to do 
this we first need to ``linearize" quantum groups (in the same spirit that, via the exponential map, complex simply connected simple 
Lie groups can be studied/classified by looking at their Lie algebras). The linearization problem is an extremely technical 
construction brought forward as a conjecture in the work of Drinfeld (\cite{DUN}), and proved in the seminal work of  Etingof and  Kazhdan (see \cite{EK1} and \cite{EK2}). An outline of this correspondence can be found in the Introductions of [KPPS1,3], wherein one can also find an explanation of why the description of which Lie bialgebras structures exists on the Lie algebra $\fg \otimes_k k((t)),$ with $\fg$ simple finite dimensional over an algebraically closed field $k$ of characteristic $0,$ arise naturally in the classification of quantum groups. The approach to the classification of Lie bialgebra structures  on $\fg \otimes_k k((t)) $ developed in [KKPS1,2,3] is by the introduction of the so-called ``Belavin--Drinfeld cohomologies". The calculation of these cohomologies is mostly done on a case-by-case basis in the classical types using  realizations of the relevant objects as matrices. The main thrust of the present paper is to realize Belavin--Drinfeld cohomologies as usual Galois cohomologies. This allows for uniform realization-free proofs in all types of results that were conjectured (and were known to hold on many of the classical types). The methods that we describe also open an avenue for further studies of Lie bialgebra structures over non-algebraically closed fields.
\medskip

\noindent {\bf Acknowledgement}. We would like to thank Seidon Alsaody for his careful reading of our manuscript.

\section{Notation}

Throughout this paper $\mathbb{K}$ will denote a field of characteristic $0$.  We fix an  algebraic closure of $\KK$ which will be denoted by $\overline{\KK}.$ The (absolute) Galois  group of the extension $\overline{\KK}/\KK$ will be denoted by $\mathcal{G}.$\footnote{For the ``untwisted" Belavin--Drinfeld cohomologies $\KK$ will be arbitrary. In the ``twisted" case $\KK = k((t))$ where $k$ is algebraically closed.}

If $V$ is a $\KK$-space (resp. Lie algebra), we will denote the $\overline{\KK}$-space (resp. Lie algebra) $V \otimes_\KK \overline{\KK}$ by $\overline{V}.$

If $\text{\bf K}$ is a linear algebraic group over $\KK$ the corresponding (non-abelian) Galois cohomology will be denoted by $H^1(\KK, \text{\bf K}).$ (See \cite{Se} for details. See also \cite{DG}, \cite{M} and \cite{SGA1} for some of the more technical aspects of this theory that will be used in what follows without further reference). We recall that $H^1(\KK, \text{\bf K}) $ coincides with the usual non-abelian continuous cohomology of  the profinite group $\mathcal{G}$ acting (naturally)  on $\text{\bf K}(\overline{\KK}).$

Let  $\fg$ be a split finite dimensional simple Lie algebra over $\mathbb{K}.$ In what follows $\bG$ will denote a split (connected) reductive algebraic group over $\KK$ with the property that the Lie algebra of the corresponding adjoint group $\bG_{\rm ad}$ is isomorphic to $\fg.$\footnote{The case which is most of interest to us is when $\bG = \bG_{\rm ad}.$ That said, peculiar phenomena appear when $\bG$ is either $\bGL_n$ or $\bSL_n$. Of course $\bG_{\rm ad}$ is then $\bPGL_n$ and $\fg = \mathfrak{sl}_n.$ The case of $\bG  = \text{\rm \bf SO}_{2n}$ is also interesting. For all of these reasons we try to maintain our  set up as general as possible.}

We fix once and for all a Killing couple $(\bB, \bH)$ of $\bG$. The induced Killing couple on $\bG_{\rm ad}$, which we denote by $(\bB_{\rm ad}, \bH_{\rm ad})$,  leads to a Borel subalgebra and split Cartan subalgebras of $\fg$ which will be denoted by $\fb$ and $\fh$ respectively. Our fixed Killing couple leads, both at the level of $\bG_{\rm ad}$ and $\fg,$ to a root system $\Delta$ with a fixed 
set of positive roots $\Delta_+$ and base $\Gamma = \{ \alpha_1, \cdots, \alpha_n \}.$\footnote{The elements of $\Delta$ are to be thought as characters of $\bH_{\rm ad}$ or elements of $\fh^*$ depending on whether we are working at the group or Lie algebra level. This will always be clear from the context.}

The Lie bialgebra structures that we will be dealing with are defined by $r$-matrices, which are  element of $\fg \otimes_\KK \fg$ satisfying  ${\rm CYB}(r) = 0$ where CYB is the classical Yang-Baxter equation (see \S3 below and \cite{ES} for definitions). For future use we introduce some terminology and notation. Consider the  action of $\bG$ on $\fg \otimes_\KK \fg$ induced from the adjoint action of $\bG$ on $\fg.$  Let $R$ be a commutative ring extension of $\KK$. If  $X \in \bG(R)$  and $ v \in (\fg \otimes_\KK\fg)_a(R) = (\fg \otimes_\KK \fg)\otimes_\KK  R \simeq (\fg \otimes_\KK R) \otimes_R (\fg \otimes_\KK R)$, then the adjoint action of $X$ in $v$ will be denoted by $\Ad_X(v).$\footnote{In contrast to the notation $(\Ad_X \otimes \Ad_X)(v)$ used elsewhere.}

Along similar lines if $\sigma \in \cal{G}$ we will write $\sigma(r)$ instead of $(\sigma \otimes \sigma)(r).$

\section{ The Belavin--Drinfeld classification. }

We maintain all of the above notation. Consider a Lie bialgebra structure $(\fg, \delta)$  on $\fg.$ By Whitehead's Lemma the cocycle $\delta : \fg \to \fg \otimes_\KK \fg$  is a coboundary. 
Thus $ \delta = \delta_r$ for some element $r \in \fg \otimes_\KK \fg$, namely
$$
\delta (a)=[a\otimes 1+1\otimes a, r]
$$
for all $a \in \fg.$ It is well-known when an element $ r \in \fg \otimes_\KK \fg$ determines a Lie bialgebra structure of $\fg. $ See \cite{ES} for details.

Assume until further notice that $\KK$ is algebraically closed. We then have the Belavin--Drinfeld classification \cite{BD}, 
which we now recall. Define an equivalence relation on $\fg \otimes_\KK \fg$ by declaring that $r$ is equivalent to $r'$ if there exist an element $X \in \bG_{\rm ad}(\KK)$ and a scalar $c \in \KK^\times$ such that
\begin{equation}\label{equivalent}
r' = c \, \Ad_X(r)
\end{equation}
If furthermore $c = 1$ the two elements are called gauge equivalent. 

Belavin--Drinfeld  provides us with a list of elements $r_{\rm BD} \in \fg \otimes_\KK \fg$ (called Beladin-Drinfeld r-matrices) with the following properties:
\begin{enumerate}
\item Each $r_{\rm BD}$ is an $r$-matrix (i.e. a solution of the classical Yang-Baxter equation) satisfying  $r + r^{21} = \Omega$ (where  $\Omega$ is the Casimir operator  of $\fg.$)

\item Any non-skewsymetric $r$-matrix for $\fg$ is equivalent to a unique $r_{\rm BD}.$

\end{enumerate}

For the reader's convenience we recall the nature of the Belavin--Drinfeld $r$-matrices. 
With respect to our fixed  $(\fb, \fh)$, any $r_{\rm BD}$ depends on a discrete and a continuous parameter. 
The discrete parameter is an admissible triple $(\Gamma_{1},\Gamma_{2},\tau)$, i.e.
an isometry $\tau:\Gamma_{1}\longrightarrow \Gamma_{2}$ where $\Gamma_{1},\Gamma_{2}\subset\Gamma$ such that
for any $\alpha\in\Gamma_{1}$ there exists $k\in \mathbb{N}$ satisfying
$\tau^{k}(\alpha)\notin \Gamma_{1}$. The continuous parameter is a tensor $r_{0}\in \mathfrak{h} \otimes_\KK\mathfrak{h}$ satisfying $r_{0}+r_{0}^{21}=\Omega_{0}$
and $(\tau(\alpha)\otimes 1+1 \otimes \alpha)(r_{0})=0$ for any $\alpha\in \Gamma_{1}$. 
Here $\Omega_{0}$ denotes the Cartan part of the quadratic Casimir element $\Omega$.
Then \[r_{\rm BD} =r_{0}+\sum_{\alpha>0}e_{\alpha}\otimes e_{-\alpha}+\sum_{\alpha\in (Span \Gamma_{1})^{+} }\sum_{k\in \mathbb{N}} e_{\alpha}\wedge e_{-\tau^{k}(\alpha)}.\]

We now return to the case of our general $\KK.$ Let $(\fg, \delta)$ be a 
Lie bialgebra structure on $\fg$. We will assume that $(\fg, \delta)$ is 
not triangular, i.e. $\delta = \delta_r$ where $r \in \fg \otimes_\KK \fg$  is not skew-symmetric. 
We view $r$ as an element of $\overline{\fg} \otimes_{\overline{\KK}} \overline{\fg}$ in the natural way and denote it 
by $\overline{r}.$ The $\overline{\KK}$ Lie bialgebra  $(\overline{\fg}, \overline{\delta})$ obtained 
by base change is given by the $r$-matrix $\overline{r}.$ By the Belavin--Drinfeld classification there exists a unique  $r_{\rm BD}$ such that  
\begin{equation}\label{crational}
\overline{r} = c \, \mathrm{Ad}_{X}(r_{\rm BD})
\end{equation}
 for some 
 $X\in \bG(\overline{\KK})$  and $c \in \overline{\KK}^\times.$ Since $\overline{r} + \overline{r}^{21} = c\,\Omega$ we can apply \cite{KKPS3} Theorem 2.7 to conclude that $c^2 \in \KK.$ 
 
 This leads to two cases, according to whether $c$ is in $\KK$ or not. The first case is treated with the untwisted Belavin--Drinfeld cohomologies, while the second one, in the case when $\KK = k((t))$ with $k$ algebraically closed of characteristic $0$, leads to twisted Belavin--Drinfeld cohomologies. These and their relations to Galois cohomology are the contents of the next two sections.
 
 \section{Untwisted Belavin--Drinfeld cohomology}

 Assume that in (\ref{crational}) we have $c \in \KK^\times.$ Let $s = c^{-1}\overline{r}$. By (\ref{crational}) 
$r_{\rm BD} = \Ad_{X^{-1}} s.$  For any element $\gamma \in \mathcal{G} = \text{\rm Gal}(\overline{\mathbb{K}}/\mathbb{K})$ 
we have $\gamma(s) = s$ and therefore $s = \Ad_{\gamma(X)}\gamma(r_{\rm BD}). $ From the foregoing it follows that
 \begin{equation}\label{keyeqinK}
 r_{\rm BD} = \Ad_{X^{-1}\gamma(X)}\gamma(r_{\rm BD})
 \end{equation}
 We can now appeal to Theorem 3 of  \cite{KKPS1} to conclude that.

\begin{theorem}\label{CartanF} Assume that $\overline{r} = c \, \mathrm{Ad}_{X}(r_{\rm BD})$ are as above.
Then $r_{\rm BD}$ is rational, i.e. it belongs to $\fg \otimes_\KK \fg.$ Furthermore $X^{-1}\gamma(X)\in \bC(\bG, r_\text{\rm BD})(\overline{\KK})$ for all $\gamma \in \mathcal{G}.$ \qed
\end{theorem}

We now recall (with our notation) the Belavin--Drinfeld cohomology definitions and 
results developed in \cite{KKPS1}.  Let $r_{\rm BD} \in \fg \otimes_\KK \fg$ be a Belavin--Drinfeld $r$-matrix.
 
\begin{definition}
  An element $X\in \bG (\overline{\mathbb{\KK}})$ is called a \emph{Belavin--Drinfeld cocycle}
associated to $\bG$ and $r_{\rm BD}$ if $X^{-1}\gamma(X)\in  \bC(\bG, r_{\rm BD})(\overline{\KK}) $, for any $\gamma \in \cal{G}.$

\end{definition}
The set of Belavin--Drinfeld cocycles associated to $r_{\rm BD}$ will be denoted by
$Z_{BD}(\bG,r_{\rm BD})$. Note that this set contains the identity element of $\bG(\overline{\KK})$.

\begin{definition}
 Two cocycles $X_1$ and $X_{2}$ in $Z_{BD}(\bG,r_{\rm BD})$ are called \emph{equivalent} if
there exists $Q\in \bG(\KK)$ and $C\in \bC(\bG, r_{\rm BD})(\overline{\KK})$ such that $X_{1}=QX_{2}C$.
\end{definition}
It is easy to check that the above defines an equivalence relation in the non-empty set $Z_{BD}(\bG,r_{\rm BD})$

\begin{definition}
Let $H_{BD}^{1}(\bG,r_{\rm BD})$ denote the set of equivalence classes of cocycles 
in $Z_{BD}(\bG ,r_{\rm BD})$. 

We call this set the \emph{Belavin--Drinfeld cohomology} associated to $(\bG, r_{\rm BD}).$
The Belavin--Drinfeld cohomology is said to be \emph{trivial} if all cocycles are equivalent to the identity, and
\emph{non-trivial} otherwise.
\end{definition}

\begin{remark} {\rm The relevance of this concept, as explained in \cite{KKPS1}, is that
there exists a one-to-one correspondence between $H^{1}_{BD}(\bG,r_{\rm BD})$
and  Lie bialgebra structures $(\fg, \delta)$ on 
$\fg$ with classical double isomorphic to 
$\fg \oplus \fg$ and  $\overline{\delta}=\delta_{r_{\rm BD}}$ up to equivalence.}
\end{remark}

Our next goal is  to realize $H_{BD}^{1}(\bG,r_{\rm BD})$ in terms of usual Galois cohomology. 
This will allow us to establish some open conjectures, as well as ``interpret" some peculiarities observed with $H_{BD}^{1}(\bG,r_\text{\rm BD})$ for certain special orthogonal groups.

\begin{proposition}\label{mainuntwistedprop}
There is a natural injection of pointed sets
$$
H_{BD}^1 (\bG,r_{BD})\to H^1 (\mathbb{K},\bC(\bG, r_{\rm BD}))
$$
\end{proposition}
\begin{proof}
 Let $X \in \bG(\KK)$ be a Belavin--Drinfeld cocycle. For $\gamma\in \cal{G}$ define 
$$u_X : \cal{G} \to \bG(\KK)$$
by 
$$
u_X:  \gamma \to u_X(\gamma) :=X^{-1} \gamma (X).
$$
Clearly $u_X$ satisfies the cocycle condition (it is in fact a cohomologically trivial 
element of $Z^1(\KK, \bG)).$ Since by definition $\gamma(X) = XC$ for some element 
$C \in \bC(\bG, r_{BD})(\overline{\KK})$, the cocycle $u_X$ takes values in 
$Z^1 (\mathbb{K},\bC(\bG, r_{BD})).$\footnote{As the reader has probably guessed, 
it will not necessarily be true that the class of our cocycle will any longer be trivial 
when viewed as taking values in the smaller group $\bC(\bG, r_{BD})$. This subtlety is  in fact the reason that allows Galois cohomology to be brought into be picture.}
 By  considering its cohomology class we obtain a map
$$ Z_{BD}^1 (\bG,r_{BD})\to H^1 (\mathbb{K},\bC(\bG, r_{BD})).$$ 
It remains to show that if $X$ and $Y$ are Belavin--Drinfeld cocycles, then $u_X$ is cohomologous $u_Y$ if and only if $X$ is equivalent to $Y.$ 

Assume that $X$ and $Y$ are equivalent. Then  $Y=QXC$ with $C\in \bC(\bG, r_{BD})(\overline{\KK})$ and
$Q\in \bG(\KK)$. Since $\gamma (Q)=Q$ for any $\gamma\in \mathcal{G}$, it follows that
$u_Y (\gamma )=C^{-1} u_X (\gamma ) \gamma (C)$, which means that $u_X$ and $u_Y$ are
cohomologous.

Conversely, if $u_X$ and $u_Y$ are
cohomologous as elements of 

\noindent $Z^1 (\mathbb{K}, \bC(\bG, r_{BD}))$ there exists 
 $C\in \bC(\bG, r_{\rm BD})(\overline{\KK})$ such 
that
$$
Y^{-1} \gamma (Y)=C^{-1} X^{-1} \gamma (X) \gamma (C)
$$
for all  $\gamma\in \mathcal{G}$.
It follows that $Q^{-1}=XCY^{-1}\in \bG(\mathbb{K})$.
This completes the proof  of the proposition. \qed
\end{proof}

The remarkable fact is that the the algebraic $\KK$-group $\bC(\bG, r_{BD}))$ is diagonalizable. Indeed since $r_{\rm BD} \in \fg \otimes_\KK \fg$ we can reason exactly as in \cite{KKPS1} Theorem 1 to conclude that.

\begin{theorem}\label{CisToral} $\bC(\bG, r_{BD})$ is a closed subgroup of $\bH.$
\qed
\end{theorem}

Combining this last result with Proposition \ref{mainuntwistedprop} we obtain, with the aid of Hilbert's theorem 90, that

\begin{corollary}\label{connected} If the algebraic $\KK$-group $\bC(\bG, r_{BD})$ is connected then 

$H_{BD}^1 (\bG,r_{\rm BD})$ is trivial. \qed
\end{corollary}

One of the most important $r$-matrices is the so-called Drinfeld--Jimbo $r_{\rm DJ}$ given by

\begin{definition}\label{rDJdefinition}
$r_\text{\rm DJ} = \sum_{\alpha>0}e_{\alpha}\otimes e_{-\alpha} + \frac{1}{2}\, \Omega_0$
\end{definition}
where $\Omega_0$ , as has already been mentioned, stands for the $\fh \otimes_\KK \fh $ component of the Casimir operator $\Omega$ of $\fg$ written with respect to our choice of $(\fb,\fh).$

In [KKPS1] it was conjectured that $H_{BD}^1 (\bG,r_{\rm DJ})$ is trivial under the assumption that $\bG$ be simple and $\mathbb{K}=\mathbb{C} ((\hbar ))$. 
The conjecture was established by a case-by-case reasoning for most of the classical groups. Further progress on this problem 
(still for the classical algebras but now with an arbitrary base field of characteristic $0$) is given in \cite{KKPS3}. 
The Galois cohomology interpretation we have given provides an affirmative much more general answer to this question. 

\begin{theorem}
$H_{BD}^1 (\bG,r_{\rm DJ})$ is trivial for any split reductive group $\bG$ over a field $\mathbb{K}$ of characteristic $0.$

\end{theorem}
\begin{proof}
We already know that $\bC(\bG, r_{DJ})$ is a closed subgroup of our split torus $\bH$. 
It is also clear from (\ref{rDJdefinition}) that all elements of $\bH(\overline{\KK})$ fix $\overline{r_{\rm DJ}}$. 
This yields $\bC(\bG, r_{DJ}) = \bH.$ By the last Corollary the Theorem follows.
\end{proof}

\begin{remark}
{\rm Since $\bC(\bG, r_{\rm BD})$ is a closed subgroup of $\bH$ it is of the form$$
\bC(\bG, r_{\rm BD}) = \bT \times \bmu_{m_1} \times \cdots \times \bmu_{m_n}
$$
where $\bT$ is a split torus over $\KK$ and  $\bmu_{m} $ is the finite multiplicative $\KK$-group of $m$--roots of unity.

Thus
$$
H^1 (\mathbb{K},\bC(\bG, r_{\rm BD})) = \mathbb{K}^{\times}/ (\mathbb{K}^{\times})^{m_1}\times \cdots \times \mathbb{K}^{\times}/ (\mathbb{K}^{\times})^{m_n}
$$

It is possible to deduce from the results of [SP, KKSP1,2,3] that for $\bG=\bGL(n)$, $\text{\bf SO}(2n+1)$, $\text{\bf Sp}(n)$ that $H_{BD}^1 (\bG,r_{\rm BD})$ is trivial. 
Though  the centralizer of Belavin--Drinfeld $r$-matrices were not explicitly computed in these papers, it  is natural to conjecture that that they are always {\it connected.} 
If so, then Corollary \ref{connected} 
would {\it show} that the corresponding $H^1_{BD}$ is trivial. This approach is not only sensible, but likely the only reasonable way of attacking the problem in  the exceptional types.

The situation for $\bG=\text{\bf SO}(2n)$ is different. Assume that $\alpha_n$ and $\alpha_{n-1}$ are the end vertices
of the Dynkin diagram of $\mathfrak{so}(2n)$. Assume also $\alpha_{n-1}=\tau^k (\alpha_{n})$ for some integer $k$,
where $\tau : \Gamma_1 \to\Gamma_2$ defines $r_{\rm BD}$. It was shown in [KKSP1] that
$\bC(\bG, r_{\rm BD})=\bT \times \Bbb{Z}/2\Bbb{Z}$ in this case and $\bC(\bG , r_{\rm BD})=\bT $ otherwise.

From our results it follows that $H_{BD}^1 (\bG,r_{\rm BD})$ is trivial in the second case.

Since $H^1 (\mathbb{K}, \bC(\bG r_{\rm BD}))=\mathbb{K}^{\times}/ (\mathbb{K}^{\times})^2$ in the first case,
to prove that the corresponding $H^1 (\text{\bf SO}(2n), r_{BD})$ is isomorphic to 
$\mathbb{K}^{\times} / (\mathbb{K}^{\times})^2$
it is sufficient to construct a non-trivial cocycle for any non-square $d\in \mathbb{K}$.
It is not difficult to see that such a cocycle can be defined by means of the element
$$
diag (d_1,d_2,...,d_{2n})\in \text{\bf SO}(2n)
$$
with $d_1=d_2=...=d_{n-1}=d_{n+2}=...=1$ and $d_n=d_{n+1}=d^{1/2}$.

We see again that the Galois cohomology point of view ``explains" why certain Belavin--Drinfeld cohomolgies are trivial, and why in the case of ${\bf SO}_{2n}$ the appearance of non-trivial classes is natural.}

\end{remark}

We end this section with a statement, which provides a complete description of
non-twisted Belavin--Drinfeld cohomologies in terms of the Galois cohomologies of 
algebraic groups.

\begin{theorem} Let $\bG$ be a split reductive group over a field $\KK$ of characteristic $0.$ 
Assume that the Lie algebra $\fg$ of the adjoint group of $\bG$ is simple. 
For any Belavin--Drinfeld $r$-matrix $r_{BD}$ in $\fg \otimes_\KK \fg$ the  sequence
$$
1 \to H_{BD}^1 (\bG,r_{\rm BD})\to H^1 (\mathbb{K} , \bC(\bG, r_{\rm BD}))\to H^1 (\mathbb{K},\bG)
$$
is exact.
\end{theorem}
\begin{proof}
This is a direct consequence of the various definitions and of Proposition \ref{mainuntwistedprop} (both the statement and the proof). \qed
\end{proof}

\medskip

From Steinberg's theorem (see \cite{Se} Ch III Theorem 3.2.1') we obtain

\begin{corollary} Assume that $\KK$ is of cohomological dimension 1.\footnote{For example $\KK = \Bbb{C}((t)).$ This is the case most relevant to quantum groups.} Then
$$
H_{BD}^1 (\bG,r_{BD}) = H^1 (\mathbb{K} , \bC(\bG, r_{BD}))
$$
\end{corollary}

\section{Twisted Belavin--Drinfeld cohomologies}

In this section we assume that $\mathbb{K}=k((t))$ where $k$ is algebraically closed of characteristic $0.$
Fix an element $j \in \overline{\KK}$ such that $j^2 = t.$ We will denote the quadratic extension $\mathbb{K}(j)$ of $\KK$  by $\LL$. 
Twisted Belavin--Drinfeld cohomologies where introduced in \cite{KKPS1} and \cite{KKPS3}
to describe a new class of Lie bialgebras structure on $\fg$ whose Drinfeld double (see \cite{ES} for the definition and constriction of this object) is isomorphic to $\fg \otimes_\KK \LL.$ 

In this section our reductive group $\bG$ will be assumed to be of adjoint type. Within the general framework described in \S 2, 
our analysis corresponds to the case  when in (\ref{crational}) the constant $c$ does not belong to $\KK$. As we have seen, then $c^2 \in \KK.$ 

Before we recall how these Lie bialgebras appear and what the relevant definitions are, we introduce some notation
 and give an explicit description of  $\text{\rm Gal}(\KK)$ and $\text{\rm Gal}(\LL)$ that will be used in the proofs.

Fix a compatible set of primitive 
$m^{\rm th}$ roots of unity $\xi_m ,$ namely such that  $\xi _{me} ^e = \xi_m
$ for all $e > 0.$ Fix also, with the obvious meaning, a compatible set $t^\frac{1}{m}$ 
of $m^{\rm th}$ roots of $t$ in $\overline{\KK}.$ There is no loss of generality in assuming that $t^\frac{1}{2} = j.$ 

Let $\KK_{m} =  \Bbb{C}((t^\frac{1}{m})).$ We can then identify $\text{\rm Gal}(\KK_m/\KK)$ 
with $\Z/m\Z $ where for each $e \in \Z$ the corresponding element $\ol{e} \in \Z/m\Z$  acts on $\KK_m$
via $ ^{\ol {e}} t^{\frac{1}{m}}_i = \xi  ^{e}_{m}
t^{\frac{1}{m}}_i.$
\smallskip

We have $\overline{\KK}  = {\limind} \,\, \KK_m.$ The absolute Galois group $\text{\rm Gal}(\KK)$ is the 
profinite completion $\widehat{\Z}$ thought as the inverse limit of the Galois groups $\text{\rm Gal}(\KK_m/\KK)$ 
as described above. It will henceforth be denoted by $\mathcal{G}$ as per our convention. If $\gamma_1$ denotes the standard profinite 
generator of $\widehat{\Z}$, then the action of $\gamma$ on $\overline{\KK}$ is given by
$$^{\gamma_1}t^\frac{1}{m} = \xi_mt^\frac{1}{m}$$
Note for future reference that $\gamma_2 := 2\gamma_1$ is the canonical profinite generator of $\text{\rm Gal}(\LL)$
\subsection{Definition of the twisted cohomologies}
Twisted cohomologies are a tool  in the  study of Lie bialgebra structures on $\fg$ such that
$$
\delta (x)=[x\otimes 1+1\otimes x, r], \ x\in \fg
$$
with an r-matrix $r$ satisfying condition $r+r^{21}=j\Omega.$\footnote{We are in the situation when $c$ in (\ref{equivalent}) is not in $\KK.$ Strictly speaking we should have $c = aj$ with $a \in \KK^\times$. Since we are working on Lie bialgebras up to equivalence we may assume  without loss of generality  that
$a = 1$.} 

The following result is proved in [KKSP1].

\begin{proposition}

Lie bialgebra structures on $\fg = \mathfrak{sl}_n$ such that the corresponding double
is isomorphic to $\fg \otimes_\KK \mathbb{L}$ are given by the formula
$$
\delta (a)=[a\otimes 1+1\otimes a, r]
$$
where $r$ satisfies $r+r^{21} =j\Omega$ and $CYB(r)=0.$

Furthermore there exists a (unique)  $r$-matrix $r_{\rm BD}$ from the Belavin--Drinfeld list of $\fg$ and an element $X \in \bG(\overline{\KK})$ such that

$ (i) \,\,  r=j Ad_X(r_{\rm BD})$

$ (ii \,a) \,\, X^{-1}\gamma(X)\in \bC(\bG, r) \,\, \text{for any} \,\, \gamma\in \text{\rm Gal}(\mathbb{L})$

$(ii \,b) \,\Ad_{X^{-1}\gamma_1 (X)}(r_{\rm BD})=r_{\rm BD}^{21}.$
\end{proposition}

To define twisted Belavin--Drinfeld cohomology we will need the following more general
\begin{proposition}

Let $r \in \overline{\fg}\otimes_{\overline{\KK}} \overline{\fg}$ be an r-matrix which defines 
a Lie bialgebra structure on $\fg$ and such that $r+r^{21} =j\Omega$.
Then 
\begin{itemize}
\item  $\gamma (r)=r$ for all $\gamma\in {\rm Gal}(\mathbb{L})$

\item $\gamma_1 (r)=-r^{21}$
\end{itemize}
%
%
%
\end{proposition}
\begin{proof}
Let $\gamma\in \Gal.$ It was proved in \cite{KKPS3} that 
\begin{equation}\label{gamma1}
  \gamma (r)=r, \,\, \text{\rm or}
  \end{equation} 
  \begin{equation}\label{gamma2} 
  \gamma (r)=r-j\Omega.
  \end{equation}
Let $G \subset \mathcal{G}$ be the subgroup of elements satisfying (\ref{gamma1}). Clearly, $G$ is a proper subgroup because  $r+r^{21} =j\Omega.$

Let $\gamma$ and $\gamma'$ satisfy (\ref{gamma2}). Then $\gamma \gamma' \in G$. It follows that $G$ is a subgroup of index $2$ and in fact $G= {\rm Gal}(\mathbb{L})$.
For $\gamma_1$ we conclude that $\gamma_1 (r)=r-j\Omega =-r^{21}$.\qed
\end{proof}

\begin{remark}
{\rm It is easy to see that if $r$ satisfies the conclusions of the proposition above, then $r$ induces a Lie bialgebra
structure on $\fg$.}
\end{remark}

Since $r+r^{21} =j\Omega$, it is clear that $r=j \Ad_X (r_{\rm BD})$ for some  $X \in \bG(\overline{\KK})$. Assume that $r_{\rm BD}$ is {\it rational}, i.e. it satisfies
$\gamma (r_{\rm BD})=r_{\rm BD}$ for all $\gamma\in \mathcal{G}$. Then we get the following two equations for $X$:

\begin{itemize}
\item $  \,\, X^{-1}\gamma(X)\in \bC(\bG, r)(\overline{\KK})\,\, \text{for any} \,\, \gamma\in \text{\rm Gal}(\mathbb{L})$

\item $ \,\Ad_{X^{-1}\gamma_1 (X)}(r_{\rm BD})=r_{\rm BD}^{21}.$
\end{itemize}
\begin{definition}
An element $X \in \bG(\overline{\KK})$ is called a twisted Belavin--Drinfeld 
cocycle for  $\bG$ and $r_{\rm BD}$ if $X^{-1}\gamma (X) \in \bC(\bG, r_{\rm BD})$ for any $\gamma\in Gal(\overline{\mathbb{K}}/\mathbb{L})$
and $\Ad_{X^{-1}\gamma_1(X)}(r_{\rm BD}) = r_{\rm BD}^{21}$. 
\end{definition}

\begin{definition}
Two twisted Belavin--Drinfeld cocycles $X$ and $Y$ are said to be equivalent  if $Y=QXC$ for some $C \in \bC(\bG, r_{\rm BD})(\overline{\KK})$and $Q\in \bG(\mathbb{K})$.
\end{definition}
It is clear that the above defines an equivalence relation on the set  $\overline{Z}^1_{BD} (\bG , r_{\rm BD})$ of twisted Belavin--Drinfeld cocycles. \begin{definition}
The twisted Belavin--Drinfeld cohomology related to $\bG$ and $r_{\rm BD}$ is the set of  equivalence classes
of the twisted cocycles. We will denote it by $\overline{H}^1_{BD} (\bG, r_{\rm BD})$.
\end{definition}

Note that it is not clear that twisted Belavin--Drinfeld cocycles exists.

\begin{remark}
{\rm Assume that $r_{\rm BD}$ is rational. Then  the twisted Belavin--Drinfeld cohomology  $\overline{H}^1_{BD} (\bG, r_{\rm BD})$
gives a one-to-one correspondence between equivalence of  Lie bialgebra structures on $\fg$
such that over $\overline{\mathbb{K}}$ they become gauge equivalent to the Lie bialgebra structure defined by $jr_{\rm BD}.$}
\end{remark}

\subsection{Twisted cohomology for the Drinfeld--Jimbo $r$-matrix}

The only good understanding of twisted Belavin--Drinfeld cohomologies is for the Drinfeld--Jimbo $r$-matrix   $r_{\rm DJ}$, which is clearly rational.
Our main goal is to establish the following.

\begin{theorem}\label{maintwisted}
The set $\overline{H}_{\rm BD}^1 (\bG, r_{\rm DJ})$ consists of one element. 
\end{theorem}

This result was  established in [KKPS1,3] for the classical Lie algebras. The key to the proof is the existence of
special elements $S\in \bG(\KK) $ and $J \in \bG(\mathbb{L})$ with the property
$$
\Ad_S (r_{DJ}) = r_{DJ}^{21} \,\, \text{\rm and} \,\, J^{-1}\gamma_1(J) = S.
$$
The existence of these elements is established by a laborious case-by-case analysis (realizing the classical algebras/groups as matrices). 
We shall provide a uniform and calculation-free proof of the existence of these elements using Steinberg's theorem (``Serre Conjecture {\rm I}"). We will then relate $\overline{H}_{\rm BD}^1$ to Galois cohomology to establish Theorem \label{maintwisted} for all types.

\subsubsection{Construction of $S$ and $J\in \bG(\mathbb{L})$ such that $\gamma_1(J)=JS$}

Let ${\rm Out}(\fg )$ be the finite group of automorphisms of the Coxeter--Dynkin diagram of our simple Lie algebra $\fg.$ If $\bOut(\fg)$  
is the corresponding constant $\KK$-group we know \cite{SGA3} that we have a split exact sequence of algebraic $\KK$-groups
\begin{equation}\label{split}
1 \to \bG \to \bAut(\fg) \to \bOut(\fg) \to 1
\end{equation}
We fix a section $\bOut(\fg) \to \bAut(\fg)$ that stabilizes $(\bB, \bH).$  This gives a  copy of ${\rm Out} (\fg ) = \bOut(\fg)(\KK)$ inside ${\rm Aut} (\fg ) := \bAut(\fg)(\KK)$ that permutes the fundamental root spaces $\fg^{\alpha_i}$ around, and which stabilizes both of our chosen  Borel and Cartan subalgebras. Of course ${\rm Aut} (\fg )$ is the semi-direct product of $\bG(\KK)$ and ${\rm Out} (\fg )$.

\begin{lemma}
Let $w_0$ be the longest element of the Weyl group $W$ of the pair $(\bB,\bH)$.
Then there exists an element $S\in \bG(\mathbb{K})$ such that $S^2=id$ and
$S(\fg^{\alpha})=\fg^{w_0 (\alpha)}$ for all roots $\alpha\in\Delta$
\end{lemma}
\begin{proof} Let $c\in {\rm Aut}(\fg )$ be the Chevalley involution. Thus $c^2=id,\ c(\fg^{\alpha})=\fg^{-\alpha}$
and $c$ restricted to the Cartan subalgebra $\fh$ is scalar multiplication by $-1$.
If ${\rm Out} (\fg)$ is trivial, then $w_0({\alpha})=-\alpha$ and we take $S = c$. 

In general note that $-w_0\in {\rm Out} (\fg)$, so we can view this as an element $d\in {\rm Aut}(\fg)$ of order 2.
Clearly, $cd=dc$ and we set $S=cd$, which is of order 2.

It remains to be shown that $S\in \bG(\mathbb{K})$. Since both $c$ and $d$ stabilize $\fh$,
so does $S$. From this it follows that $S(\fg^\alpha)=\fg^{\theta (\alpha)}$ for some $\theta \in {\rm Aut}(\Delta)$ (the automorphism group of our root system).
It is well-known that ${\rm Aut}(\Delta)$ is a semi-direct product of $W$ and ${\rm Out} (\fg )$.
Moreover, $S\in \bG(\KK)$ if and only if the restriction of $S$ to $\fh$ is in $W$. But by construction, $\theta=w_0\in W$.\qed
\end{proof}

\medskip
It is clear from (\ref{rDJdefinition}) that $Ad_S(r_{DJ})=r_{DJ}^{21}$.  Since $\bC(\bG, r_{\rm DJ}) = \bH$
we can redefine twisted Belavin--Drinfeld cocycles  for $r_{\rm DJ}$ as follows. 

\begin{lemma}\label{RDJcocycleredefinition}
An element $X \in \bG(\overline{\KK})$ is a twisted Belavin--Drinfeld cocycle for  $\bG$ and $r_{\rm DJ}$ if and only if 

(i) $X^{-1}\gamma (X) \in \bH(\overline{\KK})$ for any $\gamma \in {\rm Gal}(\mathbb{L})$, and 

(ii) $\Ad_{X^{-1}\gamma_1 (X)}(r_{\rm BD}) = \Ad_S(r_{\rm BD})$. 
\end{lemma}

As we shall see this definition will allow us to compute the corresponding twisted Belavin--Drinfeld cohomology.

\begin{proposition}
Let $S\in \bG(\mathbb{K})$ be as in the previous lemma. Then there exists
$J\in \bG(\mathbb{L})$ such that $\gamma_1 (J)=JS$
\end{proposition}
\begin{proof}
There exists a unique continuous group homomorphism
$u: \Gal \to G(\overline{\mathbb{K}})$ such that
$u(\gamma_1) = S$. Given that $\gamma_1 (S)=S$ our $u$ is a cocycle in $Z^1 (\mathbb{K}, \bG)$.

Since $\KK$ is of cohomological dimension $1$ by  Steinberg's theorem $H^1 (\mathbb{K}, \bG)=1$.
Therefore, there exists $J\in \bG(\overline{\mathbb{K}})$ such that $J^{-1}\gamma_1 (J)=S$.
It remains to be shown that $J\in \bG(\mathbb{L})$. For this note that
$$
2\gamma_1 (J)=\gamma_1 (\gamma_1 (J))=\gamma_1 (JS)=\gamma_1 (J)S=JS^2=J
$$
Since $2\gamma_1$ pro-generates ${\rm Gal} (\mathbb{L})$ it follows that  $J\in \bG(\mathbb{L})$ as desired.\qed
\end{proof}

\medskip 
Note that our element $J$ is a twisted Belavin--Drinfeld cocycle.
\subsubsection{Computation of $\overline{H}^1_{BD} (\bG , r_{\rm DJ})$}

The aim of this section to show that $\overline{H}^1_{BD} (\bG , r_{\rm DJ})$ consists of one element
generated by the class of the element $J$ constructed above. This will in particular prove Theorem \ref{maintwisted}. 

It is clear that our element $S$ normalizes (in the functorial sense) $\bH$. We can therefore consider the $\KK$- group 
$$ \tilde{\bH} = \bH \rtimes \{1, S \}.$$ 
Strictly speaking we should be writing the constant $\KK$-group corresponding to the finite group $\{1, S\}.$ For this reason we shall also write  
$$ \tilde{\bH} = \bH \rtimes \Bbb{Z}/2\Bbb{Z}$$
where $\Bbb{Z}/2\Bbb{Z}$ acts on $\bH$ by means of $S.$ 

Let us begin by explicitly determining $H^1(\KK, \tilde{\bH}).$  Consider the split exact sequence of $\KK$ groups 
$$1 \to \bH \to \tilde{\bH} \to \Bbb{Z}/2\Bbb{Z} \to 1.$$
Passing to cohomology we get 
$$H^1(\KK, \bH) \to H^1(\KK, \tilde{\bH}) \to H^1(\KK, \Bbb{Z}/2\Bbb{Z}) \to 1.$$
The surjectivity of the last map follows from the fact that the original sequence of $\KK$-groups splits.
We have $H^1(\KK, \Bbb{Z}/2\Bbb{Z}) = \KK^\times/(\KK^\times)^2.$ This last is the group of order $2$ with representatives $\{1, j\}$ where we recall that $ j = t^\frac{1}{2}.$

The elements of $H^1(\KK, \tilde{\bH})$ mapping to the class of $1$ are given by the image of $H^1(\KK, \bH)$ 
which is trivial by Hilbert 90. The elements of $H^1(\KK, \tilde{\bH})$ mapping to the class of $j$ 
are given by the image of $H^1(\KK, \bH')$ where the $\KK$-group $\bH'$ is a twisted form of $\bH.$ 
By Steinberg's theorem $H^1(\KK, \bH')$ vanishes. It follows that $H^1(\KK, \tilde{\bH})$ has two elements. More precisely.

\begin{theorem}
The pointed set $H^1 (\mathbb{K},\bH\rtimes \{1,S\})$ consists of the two elements:
\begin{enumerate}

\item{The trivial class,}

\item{The class of the cocycle $u_J$ defined by $u_J,\ u_J(\gamma)=J^{-1} \gamma (J).$ In particular $u_J(\gamma_1) = S.$}
\end{enumerate}

\end{theorem}

If $X \in \bG(\overline{\KK})$ is a twisted Belavin--Drinfeld cocycle for $r_{\rm DJ}$ it is clear from Lemma \ref{RDJcocycleredefinition} that the map $\tilde{u}_X : {\rm Gal}(\KK) \to \bG(\overline{\KK})$ given by
$$ \tilde{u}_X : \gamma \mapsto X^{-1} \gamma(X)$$
is a Galois cohomology cocycle in $Z^1(\KK, \tilde{\bH}).$ 

\begin{theorem}
The map $X \mapsto \tilde{u}_X$ described above induces an injection
$\overline{H}^1_{BD} (\bG, r_{\rm DJ})\to H^1 (\KK, \tilde{\bH}) = \{1,j\}.$ More precisely the fiber of the trivial class $1$ is empty and that of $j$ consist of the class of the Belavin--Drinfeld cocycle $J.$
\end{theorem}
\begin{proof}
If  $X$ and $Y$ are equivalent Belavin--Drinfeld cocycle for $r_{\rm DJ}$ then by definition  $Y=QXC$ where $Q\in G(\mathbb{K})$ and $C \in \bH(\overline{\KK}).$ Just as in the untwisted case we see that the Galois cocycles $\tilde{u}_X $ and $\tilde{u}_Y $ are cohomologous. We thus have a canonical map
$$i:\overline{H}^1_{BD} (\bG, r_{\rm DJ})\to H^1  (\mathbb{K},\bH \rtimes \{1,S\})$$
We now look in detail at the two fibers. Let $X \in \bG(\overline{\KK})$ be a twisted Belavin--Drinfeld cocycle.

\begin{enumerate}
\item{ Suppose that $\tilde{u}_X$ is in the trivial class $1\in \{1,j\}$. By definition there exists an element $h \in \tilde{\bH}(\overline{\KK})$ such that  $\tilde{u}_X(\gamma) = h^{-1}{^\gamma}h.$ 
Let $C \in \bH(\overline{\KK})$ and  $\epsilon \in \{0,1\}$ be such that $h =S{^\epsilon}C.$ Since $S$ is fixed by the Galois group $h^{-1}{^\gamma}h = C^{-1}{^\gamma}C.$ But this implies, in particular, that $\tilde{u}_X(\gamma_1) \in \bH(\overline{\KK})$. This last is false since
$$\tilde{u}_X(\gamma_1)(r_{\rm DJ}) = X^{-1}\gamma_1(X)(r_{\rm DJ}) = r_{\rm DJ}^{21}\neq  r_{\rm DJ}.$$
The fiber of the trivial class $1$ under our canonical map is therefore empty.}

\item{Suppose that the class of $X$ is  mapped to $j\in \{1,j\}$. Then $\tilde{u}_X$ is cohomologous to $u_J$. By definition there exist $h =S{^\epsilon}C$ as above such that
\begin{equation}\label{twistedcocycle}
X^{-1}\gamma (X)=C^{-1}S^{\epsilon}J^{-1}\gamma(J)S^{\epsilon}\gamma (C)
\end{equation}

for all $\gamma \in \mathcal{G}.$ An arbitrary element of our Galois group is of the form $\gamma_n = n \gamma_1$ Recall that $J \in \bG(\mathbb{L})$ (hence it is fixed by all $\gamma_n$ with $n$ even),  that $J^{-1}\gamma(J) = S \in \bG(\KK)$ and that $S^2 = 1.$ These easily imply that   $J^{-1}\gamma_n(J) = S^n.$  Taking this into account we get from (\ref{twistedcocycle}) that for all $n \in \mathbb{Z}$
 \begin{equation}\label{tc} X^{-1}\gamma_n (X)=C^{-1}J^{-1}\gamma_n(J)\gamma_n (C) \,\, \text{\rm if $n$ is odd}
\end{equation}

From these it readily follows that $Q^{-1} := JCX^{-1}$ is invariant under the action of $\mathcal{G}.$ 
Thus $Q \in \bG(\KK).$ Since $X = QJC$  we have that $X$ and $J$ are equivalent Belavin--Drinfeld cocycles. 
The fiber of $j$ has therefore exactly one element.}
\end{enumerate} 

This completes the proof. \qed
\end{proof}
\medskip

This last result shows that Theorem \ref{maintwisted} holds.  More precisely. \begin{corollary}
The twisted Belavin--Drinfeld cohomology $\overline{H}^1_{BD} (\bG, r_{\rm DJ})$ consists of one class
only, namely the class of the cocycle $J$. \qed
\end{corollary}

\end{document}